\documentclass[reqno,11pt]{amsart}
\usepackage{tikz}
\usetikzlibrary{snakes}
\newtheorem{theorem}{Theorem}[section]
\newtheorem{corollary}[theorem]{Corollary}
\newtheorem{proposition}[theorem]{Proposition}

\newtheorem{lemma}[theorem]{Lemma}

\newtheorem{definition}[theorem]{Definition}

\theoremstyle{remark}
\newtheorem{remark}[theorem]{Remark}
\newtheorem{example}[theorem]{Example}

\numberwithin{equation}{section}
\allowdisplaybreaks

\newcommand{\ta}{\theta}
\newcommand{\C}{\mathbb C}

\newcommand{\naturals}{\mathbb N}
\newcommand{\Z}{\mathbb Z}

\definecolor{dblackcolor}{rgb}{0.0,0.0,0.0}
\definecolor{dredcolor}{rgb}{0.9,0.3,0.4}
\definecolor{dbluecolor}{rgb}{0.01,0.02,0.7}
\definecolor{dgreencolor}{rgb}{0.2,0.5,0.0}
\definecolor{dgraycolor}{rgb}{0.30,0.3,0.30}
\definecolor{gr}{rgb}{0.10,0.5,0.20}

\author[Michael J.\ Schlosser]{Michael J.\ Schlosser$^{*}$}
\address{Fakult\"at f\"ur Mathematik, Universit\"at Wien,
Oskar-Morgenstern-Platz~1, A-1090 Vienna, Austria}
\email{michael.schlosser@univie.ac.at}
\thanks{$^{*}$ Partially supported by FWF Austrian Science Fund
grant P32305.}

\author[Meesue Yoo]{Meesue Yoo$^{**}$}
\address{Department of Mathematics, Chungbuk National University,
Cheongju 28644, South Korea}
\email{meesueyoo@chungbuk.ac.kr}
\thanks{$^{**}$ Partially supported by the National Research Foundation of
Korea (NRF) grant funded by the Korea government
No.\! 2020R1F1A1A01064138.}

\input xy
\xyoption{all}

\title{Elliptic solutions of dynamical Lucas sequences}

\dedicatory{Dedicated to Johann Cigler on the occasion of his $84$th birthday}



\begin{document}
%------------------------------------------------------------------------------------------------

\begin{abstract}
We study two types of dynamical extensions of Lucas sequences
and give elliptic solutions for them.
The first type concerns a level-dependent
(or discrete time-dependent) version involving commuting variables.
We show that a nice solution
for this system is given by elliptic numbers.
The second type involves a non-commutative version of
Lucas sequences which defines the non-commutative (or abstract)
Fibonacci polynomials introduced by Johann Cigler.
If the non-commuting variables are specialized to be elliptic-commuting
variables the abstract Fibonacci polynomials become
non-commutative elliptic Fibonacci polynomials.
Some properties we derive for these include their explicit expansion
in terms of normalized monomials and a non-commutative elliptic
Euler--Cassini identity.
\end{abstract}

\maketitle

%------------------------------------------------------------------------------------------------
\section{Introduction}
%------------------------------------------------------------------------------------------------

In a series of papers, Lucas~\cite{L1,L2,L3}
studied the generalized Fibonacci polynomials $\langle n\rangle$
which depend on two commuting variables $P,Q$ and are defined by
$\langle 0\rangle=0$,  $\langle 1\rangle=1$, and
\begin{equation}\label{eq:lucas}
\langle n\rangle=P\langle n-1\rangle+Q\langle n-2\rangle,  
\end{equation}
for $n\ge 2$. (The two initial conditions $\langle 0\rangle=0$,  $\langle 1\rangle=1$
can be altered of course but we shall stick to them here as specified.)
For example, we have
\begin{equation*}
\langle 2\rangle=P,\quad\langle 3\rangle=P^2+Q,\quad
\langle 4\rangle=P^3+2PQ,\quad \langle 5\rangle=P^4+3P^2Q+Q^2.
\end{equation*}
For $P=Q=1$, this sequence reduces to the Fibonacci numbers
$\langle n\rangle=F_n$.
For $P=2$, $Q=-1$ it reduces to the nonnegative integers $\langle n\rangle=n$.
For $P=q+q^{-1}$, $Q=-1$, it reduces to the quantum integers 
$\langle n\rangle=\langle n\rangle_q:=\frac{q^n-q^{-n}}{q-q^{-1}}$,
while for $P=1+q$, $Q=-q$, it reduces to the (standard) $q$-integers
$\langle n\rangle=[n]_q:=\frac{1-q^n}{1-q}$.
More generally, for $P=c(1+q)$ and $Q=-c^2q$, it reduces to $c^{n-1}[n]_q$,
unifying the last two cases.

%For $n\ge 0$, the Lucas analogue of a factorial is the \emph{Lucastorial}
%\begin{equation*}
%\langle n\rangle!=\langle 1\rangle\langle 2\rangle\cdots\langle n\rangle,
%\end{equation*}
%which can be used to define, for $0\le k\le n$, the \emph{Lucasnomial}
%\begin{equation*}
%\left<\begin{matrix}n\\k\end{matrix}\right>=
%\frac{\langle n\rangle!}{\langle k\rangle!\langle n-k\rangle!}.
%\end{equation*}

%A first combinatorial interpretation for the Lucasnomials was given
%in \cite{GV}. 
%Other constructions were given in \cite{BP}, \cite{SS}, and, most recently,
%in \cite{BCMS}.

A function is defined to be {\em elliptic} if it is meromorphic and
doubly periodic. It is well known (cf.\ e.g.\ \cite{W}) that elliptic functions
can be expressed in terms of quotients of products of theta functions.
Define for $z\neq 0$ the \emph{(modified Jacobi) theta function} with
\emph{nome} $p$ by
\begin{equation*}
\theta(z;p)=\prod_{j\ge 0}\big((1-p^jz)(1-p^{j+1}/z)\big),\qquad\quad
|p|<1.
\end{equation*}
For brevity, we write
\begin{equation*}
\theta(z_1,\dots,z_m;p)=\theta(z_1;p)\cdots\theta(z_m;p)
\end{equation*}
for products of these functions.
The modified Jacobi theta functions satisfy the \emph{inversion formula}
\begin{subequations}
\begin{equation}\label{eq:inv}
\theta(z;p)=-z\theta(1/z; p),
\end{equation}
the \emph{quasi-periodicity relation}
\begin{equation}\label{eq:qpp}
\theta(pz; p)=-\frac{1}{z}\theta(z;p),
\end{equation}
and the \emph{addition formula}
\begin{equation}\label{eq:addf}
\theta(uv,u/v,wz,w/z;p)-\theta(uz,u/z,wv,w/v;p)
=\frac wv\,\theta(vz,v/z,uw,u/w;p)
\end{equation}
\end{subequations}
(cf.\ \cite[p.~451, Example 5]{WhW}).

In this paper, we study two types of dynamical extensions of
Lucas sequences and give \emph{elliptic} solutions for them.
The first type concerns a \emph{level-dependent}
(or \emph{discrete time-dependent}) version of \eqref{eq:lucas}
involving commuting variables. We show that a nice solution for
this system is given in terms of elliptic numbers. The second
type is a \emph{non-commutative} version which defines the
\emph{non-commutative} (or \emph{abstract})
\emph{Fibonacci polynomials} introduced by Johann
Cigler~\cite{C}. We study some (known and new) properties for these.
In particular, we extend the sequence of these polynomials
to negative indices and recover a formula by Cigler~\cite[Section~3]{C2}
for the negatively indexed
non-commutative Fibonacci polynomials in terms
of the non-negatively indexed ones. This allows us to
establish a non-commutative Euler--Cassini identity.
In the non-commutative setting we also take a closer
look at the case when the non-commuting variables
are specialized to satisfy \emph{weight-dependent
commutation relations}. In this case the non-commutative
Fibonacci polynomials become, what we shall call,
\emph{non-commutative weight-dependent Fibonacci polynomials}.
We show that after normal ordering of the weight-dependent-commuting
variables weight-dependent binomial coefficients appear
in the expansion of the normalized monomials.
A further specialization of interest concerns the introduction
of elliptic weights. For \emph{elliptic-commuting variables}
the non-commutative Fibonacci polynomials become, what we shall call,
\emph{non-commutative elliptic Fibonacci polynomials}.
In this case after normal ordering of the elliptic-commuting
variables fully factorized elliptic binomial coefficients appear
in the expansion of the normalized monomials.
This extends the basic case (or $q$-case) for $q$-commuting variables.
We also establish an explicit Euler--Cassini identity for the
non-commutative elliptic Fibonacci polynomials.

We would like to point out that the results in the current paper
do not appear to directly contain the \emph{elliptic Fibonacci numbers}
which were introduced in \cite{SY-com} nor those (of a simpler type)
which were introduced in \cite{BCK}. While we believe that there
is a connection of our non-commutative elliptic Fibonacci polynomials
considered in Section~\ref{sec:nc} of this paper with our earlier
elliptic Fibonacci numbers in \cite{SY-com}, the connection is
not yet entirely clear and requires further investigations.

Our paper is organized as follows.
In Section~\ref{sec:dyn} we study the level-dependent Lucas system
with commutative variables and give an elliptic solution for it.
In Section~\ref{sec:pre} we describe the algebras of
weight-dependent-commuting and elliptic-commuting variables we are
working with in the final section, and also define
corresponding weighted and elliptic binomial coefficients.
Finally, Section~\ref{sec:nc} is devoted to the non-commutative
Lucas equation and the noncommutative weight-dependent and
elliptic Fibonacci polynomials.

%------------------------------------------------------------------------------------------------
\section{Elliptic solution of a level-dependent Lucas system}\label{sec:dyn}
%------------------------------------------------------------------------------------------------

In this section, 
we consider the following level-dependent extension of Lucas'
generalized Fibonacci polynomials $\langle n\rangle$
defined by the recurrence relation \eqref{eq:lucas}. 
We consider sequences of variables $(P_\ell)_{\ell\ge 0}$
and $(Q_\ell)_{\ell\ge 0}$ (where the index $\ell$ could be
thought of being the \emph{level} or \emph{discrete time}).
Now define the doubly-indexed sequence
$\big(\langle n\rangle_\ell\big)_{n,\ell\ge 0}$
by $\langle 0\rangle_\ell=0$, $\langle 1\rangle_\ell=1$,
for all $\ell\ge 0$ and, instead of \eqref{eq:lucas},
assume the following dynamical recurrence relation:
\begin{equation}\label{eq:l-lucas}
\langle n\rangle_\ell=P_\ell\langle n-1\rangle_{\ell+1}+
Q_\ell\langle n-2\rangle_{\ell+2}, 
\end{equation}
for $n\ge 2$ and all $\ell\ge 0$. Here we have
\begin{align*}
\langle 2\rangle_\ell&=P_\ell\\
\langle 3\rangle_\ell&=P_\ell P_{\ell+1}+Q_\ell,\\
\langle 4\rangle_\ell&=P_\ell P_{\ell+1}P_{\ell+2}+P_\ell Q_{\ell+1}+
P_{\ell+2}Q_\ell,\\
\langle 5\rangle_\ell&=P_\ell P_{\ell+1}P_{\ell+2}P_{\ell+3}+
P_\ell P_{\ell+1}Q_{\ell+2}+P_\ell P_{\ell+3}Q_{\ell+1}+
P_{\ell+2}P_{\ell+3}Q_\ell+Q_\ell Q_{\ell+2},
\end{align*}
for all $l\ge 0$.

We now show that the system in \eqref{eq:l-lucas} admits
a nice solution involving elliptic functions.

Let $a$ and $b$ be two independent variables, and $q\in\C$ be the \emph{base}.
It readily follows by the addition formula \eqref{eq:addf} that for
\begin{equation}\label{eq:xlyl}
P_\ell=\frac{\ta(q^2,aq^{\ell+2},bq^{2\ell+2},aq^{-\ell}/b;p)}
{\ta(q,aq^{\ell+1},bq^{2\ell+3},aq^{1-\ell}/b;p)},
\quad\; Q_\ell=-\frac{\ta(aq^{\ell+3},bq^{2\ell+1},aq^{-1-\ell}/b;p)}
{\ta(aq^{\ell+1},bq^{2\ell+3},aq^{1-\ell}/b;p)}q
\end{equation}
the sequence defined by the system in \eqref{eq:l-lucas} reduces to
the \emph{elliptic integers} 
\begin{equation}\label{eq:ellint}
\langle n\rangle_\ell=\langle n\rangle_{aq^\ell,bq^{2\ell};q,p}:=
\frac{\theta(q^n, a q^{\ell+n}, b q^{2\ell+n}, a q^{2-\ell-n}/b;p)}
{\theta(q, aq^{\ell+1}, b q^{2\ell+2n-1}, aq^{1-\ell}/b;p)}.
\end{equation}
Indeed, if we insert $P_\ell$ and $Q_\ell$ from \eqref{eq:xlyl}
in \eqref{eq:l-lucas}, we obtain
\begin{align*}
&P_\ell\langle n-1\rangle_{\ell+1}+
Q_\ell\langle n-2\rangle_{\ell+2}\\
&=
\frac{\ta(q^2,aq^{\ell+2},bq^{2\ell+2},aq^{-\ell}/b;p)}
{\ta(q,aq^{\ell+1},bq^{2\ell+3},aq^{1-\ell}/b;p)}\,
\frac{\theta(q^{n-1}, a q^{\ell+n}, b q^{2\ell+n+1}, a q^{2-\ell-n}/b;p)}
{\theta(q, aq^{\ell+2}, b q^{2\ell+2n-1}, aq^{-\ell}/b;p)}\\
&\quad\,-\frac{\ta(aq^{\ell+3},bq^{2\ell+1},aq^{-1-\ell}/b;p)}
{\ta(aq^{\ell+1},bq^{2\ell+3},aq^{1-\ell}/b;p)}q\,
\frac{\theta(q^{n-2}, a q^{\ell+n}, b q^{2\ell+n+2}, a q^{2-\ell-n}/b;p)}
{\theta(q, aq^{\ell+3}, b q^{2\ell+2n-1}, aq^{-1-\ell}/b;p)}\\
&=\frac{\ta(q^2,bq^{2\ell+2},q^{n-1}, a q^{\ell+n}, b q^{2\ell+n+1},
a q^{2-\ell-n}/b;p)}
{\ta(q,aq^{\ell+1},bq^{2\ell+3},aq^{1-\ell}/b,q, b q^{2\ell+2n-1};p)}\\
&\quad\,-q\frac{\ta(bq^{2\ell+1},q^{n-2}, a q^{\ell+n}, b q^{2\ell+n+2},
a q^{2-\ell-n}/b;p)}
{\ta(aq^{\ell+1},bq^{2\ell+3},aq^{1-\ell}/b,q,
b q^{2\ell+2n-1};p)}\\
&=
\frac{\ta(a q^{\ell+n},a q^{2-\ell-n}/b;p)}
{\ta(q,aq^{\ell+1},bq^{2\ell+3},aq^{1-\ell}/b,q, b q^{2\ell+2n-1};p)}\\
&\quad\,\times
\left[\ta(q^2,bq^{2\ell+2},q^{n-1}, b q^{2\ell+n+1};p)-
q\,\ta(q,bq^{2\ell+1},q^{n-2},b q^{2\ell+n+2};p)\right]\\
&=\frac{\ta(a q^{\ell+n},a q^{2-\ell-n}/b;p)}
{\ta(q,aq^{\ell+1},bq^{2\ell+3},aq^{1-\ell}/b,q, b q^{2\ell+2n-1};p)}
\ta(q^n,bq^{2\ell+n},q,bq^{2\ell+3};p)\\
&=\frac{\ta(q^n,a q^{\ell+n},b q^{2\ell+n},a q^{2-\ell-n}/b;p)}
{\ta(q,aq^{\ell+1},b q^{2\ell+2n-1},aq^{1-\ell}/b;p)}=\langle n\rangle_\ell,
\end{align*}
where the difference of the two products of theta functions in the pair
of brackets in the fourth equality was simplified with respect to the
\begin{equation*}
(u,v,w,z)\mapsto\big(b^{\frac 12}q^{\ell+n},b^{\frac 12}q^{\ell+1},
b^{\frac 12}q^{\ell+2},b^{\frac 12}q^\ell\big)
\end{equation*}
case of \eqref{eq:addf}. This proves the claim about the elliptic solution.

The elliptic integers in \eqref{eq:ellint} can actually be identified as
specialized elliptic binomial coefficients
$$
\langle n\rangle_\ell=\begin{smallmatrix}\begin{bmatrix}n\\
n-1\end{bmatrix}\end{smallmatrix}_{aq^\ell,bq^{2\ell};q,p},
$$
the general case of the elliptic binomial coefficients being defined in \eqref{ellbin}.
% which in full generality were introduced in \cite[Eq.~(4.4)]{Schl1}
%and were also employed in \cite{SY-rook,SY-com}.

Finally, we point out a simple way to obtain
a new dynamical Lucas sequence from a given one
by a suitable ``scaling" of the variables with respect to an additional
sequence $(c_\ell)_{\ell\ge 0}$.
In particular, given three sequences
$$
\big(\langle n\rangle_\ell\big)_{n,\ell\ge 0},\qquad (P_\ell)_{\ell\ge 0},
\qquad (Q_\ell)_{\ell\ge 0},
$$
satisfying \eqref{eq:l-lucas} with the initial conditions
$\langle 0\rangle_\ell=0$ and $\langle 1\rangle_\ell=1$,
for all $\ell\ge 0$,  the three sequences
$$
\Big(\widetilde{\langle n\rangle}_\ell\Big)_{n,\ell\ge 0},\qquad
\big(\widetilde{P}_\ell\big)_{\ell\ge 0},
\qquad \big(\widetilde{Q}_\ell\big)_{\ell\ge 0},
$$
with the initial conditions
$\widetilde{\langle 0\rangle}_\ell=0$ and $\widetilde{\langle 1\rangle}_\ell=1$,
for all $\ell\ge 0$, also satisfy \eqref{eq:l-lucas}, where
$$
\widetilde{\langle n\rangle}_\ell=
c_{\ell}c_{\ell+1}\cdots c_{\ell+n-2} \langle n\rangle_\ell,
\qquad\widetilde{P}_\ell=c_\ell P_\ell,
\qquad\widetilde{Q}_\ell=c_\ell c_{\ell+1} Q_\ell,
$$
for all $n\ge 2$ and $\ell\ge 0$. It is straightforward to confirm this assertion
by multiplying both sides of \eqref{eq:l-lucas} with the
product $c_{\ell}c_{\ell+1}\cdots c_{\ell+n-2}$.

%------------------------------------------------------------------------------------------------
\section{Weight-dependent commutation relations and elliptic weights}\label{sec:pre}
%------------------------------------------------------------------------------------------------
%------------------------------------------------------------------------------------------------
\subsection{Noncommutative weight-dependent binomial theorem}
%------------------------------------------------------------------------------------------------
The material in this subsection, up to Lemma~\ref{lem:com},
is taken from the first author's paper \cite{Schl1}, while the material
afterwards is new.

Let $\mathbb{N}$ and $\mathbb{N}_0$ denote the sets of positive and
nonnegative integers, respectively. 

\begin{definition}\label{def:Cwxy}
For a doubly-indexed sequence of indeterminates $(w(s,t))_{s,t\in\mathbb{N}}$,
let $\mathbb{C}_w [x,y]$ be the associative unital algebra over $\mathbb{C}$
generated by $x$ and $y$, satisfying the following three relations:
\begin{subequations}\label{eqn:noncommrel}
\begin{align}
yx &= w(1,1)xy,\\
x w(s,t)&= w(s+1,t)x,\\
y w(s,t)&= w(s,t+1)y,
\end{align}
\end{subequations}
for all $s,t\in\mathbb N$. 
\end{definition}

For $s\in\mathbb{N}$ and $t\in \mathbb{N}_0$, we define 
\begin{equation}\label{eq:W}
W(s,t):= \prod_{j=1}^t w(s,j),
\end{equation}
the empty product being defined to be $1$.
Note that for $s,t\in\mathbb{N}$, we have $w(s,t)=W(s,t)/W(s,t-1)$.
We refer to the $w(s,t)$ as {\em small weights}, whereas to the
$W(s,t)$ as {\em big weights} (or {\em column weights}).

Let the \emph{weight-dependent binomial coefficients} be defined by 
\begin{subequations}\label{wbineq}
\begin{align}\label{recu}
&{}_{\stackrel{\phantom w}{\stackrel{\phantom w}w}}\!\!
\begin{bmatrix}0\\0\end{bmatrix}=1,\qquad
{}_{\stackrel{\phantom w}{\stackrel{\phantom w}w}}\!\!
\begin{bmatrix}n\\k\end{bmatrix}=0
\qquad\text{for\/ $n\in\naturals_0$,\, and\/
$k\in-\naturals$ or $k>n$},\\
\intertext{and}
\label{recw}
&{}_{\stackrel{\phantom w}{\stackrel{\phantom w}w}}\!\!
\begin{bmatrix}n+1\\k\end{bmatrix}=
{}_{\stackrel{\phantom w}{\stackrel{\phantom w}w}}\!\!
\begin{bmatrix}n\\k\end{bmatrix}
+{}_{\stackrel{\phantom w}{\stackrel{\phantom w}w}}\!\!
\begin{bmatrix}n\\k-1\end{bmatrix}
\,W(k,n+1-k)
\qquad \text{for $n,k\in\naturals_0$}.
\end{align}
\end{subequations}

These weight-dependent binomial coefficients have a
combinatorial interpretation in terms of \emph{weighted lattice paths},
see \cite{Schl0}. Here, a lattice path is a sequence of north (or vertical)
and east (or horizontal) steps in the first quadrant of the $xy$-plane,
starting at the origin $(0,0)$ and ending at say $(n,m)$.
We give weights to such paths by assigning the big weight $W(s,t)$
to each east step $(s-1,t)\rightarrow (s,t)$ and $1$ to each north step.
Then define the weight of a path $P$, $w(P)$, to be the product of the
weight of all its steps. 

Given two points $A,B\in\naturals_0^2$,
let $\mathcal{P}(A\rightarrow B)$ be the set of all lattice paths
from $A$ to $B$, and define 
\begin{displaymath}
w(\mathcal{P}(A\rightarrow B)):= \sum_{P\in \mathcal{P}(A\rightarrow B)}w(P).
\end{displaymath}
Then we have 
\begin{equation}
w(\mathcal{P}((0,0)\rightarrow (k,n-k)))=
{}_{\stackrel{\phantom w}{\stackrel{\phantom w}w}}\!\!
\begin{bmatrix}n\\k\end{bmatrix}\label{eqn:wbin}
\end{equation}
as both sides of the equation satisfy the same recursion
and initial condition as in \eqref{wbineq}.

Interpreting the $x$-variable as an east step and the $y$-variable
as a north step, we get the following weight dependent binomial theorem.

\begin{theorem}[\cite{Schl1}]\label{thm:binomial}
Let $n\in\mathbb{N}_0$. Then, as an identity in $\mathbb{C}_w[x,y]$,
\begin{equation}\label{eqn:binomial}
(x+y)^n =\sum_{k=0}^n
{}_{\stackrel{\phantom w}{\stackrel{\phantom w}w}}\!\!
\begin{bmatrix}n\\k\end{bmatrix}x^k y^{n-k}.
\end{equation}
\end{theorem}

The following rule for interchanging powers of $x$ and $y$
is easy to prove by induction (and it is also easy to interpret
combinatorially by considering weighted lattice paths);
we therefore omit the proof.
\begin{lemma}[\cite{Schl1}]\label{lem:com}
We have
$$
y^k x^\ell=\left(\prod_{i=1}^\ell\prod_{j=1}^k w(i,j)\right)x^\ell y^k
=\left(\prod_{i=1}^\ell W(i,k)\right)x^\ell y^k.
$$
\end{lemma}

We now extend the algebra $\mathbb{C}_w [x,y]$ from
Definition~\ref{def:Cwxy} to the algebra $\mathbb{C}_w [x,x^{-1},y]$:
\begin{definition}\label{def:Cwxy2}
For a doubly-indexed sequence of invertible
indeterminates\break $(w(s,t))_{s\in\mathbb{Z},t\in\mathbb{N}}$,
let $\mathbb{C}_w [x,x^{-1},y]$ be the associative unital algebra
over $\mathbb{C}$ generated by $x$, $x^{-1}$ and $y$,
satisfying the following relations:
\begin{subequations}\label{eqn:noncommrel2}
\begin{align}
x^{-1}x&=xx^{-1}=1\\
yx &= w(1,1)xy,\\
x^{-1}y&=w(0,1)yx^{-1},\\
x w(s,t)&= w(s+1,t)x,\\
x^{-1} w(s,t)&= w(s-1,t)x^{-1},\\
y w(s,t)&= w(s,t+1)y,
\end{align}
\end{subequations}
for all $s\in\mathbb Z$ and $t\in\mathbb N$. 
\end{definition}
It is easy to see that the above relations are compatible with each other
and naturally extend \eqref{eqn:noncommrel}.

The following lemma which is easy to verify will be used in Section~\ref{sec:nc}.
\begin{lemma}\label{lem:inv}
Let $(w(s,t))_{s\in\mathbb Z,t\in\mathbb{N}}$ be a doubly-indexed sequence of
invertible indeterminates, and $x$ and $y$ variables with $x$ being invertible,
together forming the associative algebra $A=\mathbb{C}_w [x,x^{-1},y]$.
Then there is an involutive algebra isomorphism
$$\phi: A\to \widetilde{A}$$
where
$$\widetilde{A}=\mathbb{C}_{\widetilde{w}} [x^{-1},x,x^{-1}y]$$
with
\begin{equation}\label{eq:tw}
\widetilde{w}(s,t)=w(1-s-t,t)^{-1}.
\end{equation}
\end{lemma}
It is indeed straightforward to check that the simultaneous replacement
of $w(s,t)$ ($s\in\mathbb Z$, $t\in\mathbb N$),
$x$ and $y$ in \eqref{eqn:noncommrel2}
by $w(1-s-t,t)^{-1}$, $x^{-1}$ and $x^{-1}y$, respectively,
again satisfies the conditions in \eqref{eqn:noncommrel2}.

As a consequence, given an identity in $w(s,t)$
($s\in\mathbb Z$, $t\in\mathbb N$), $x$ and $y$,
a new valid identity can be obtained by applying the
isomorphism $\phi$ to each of the occurring variables,
where in both identities the variables satisfy the same commutation
relations \eqref{eqn:noncommrel2}.

%------------------------------------------------------------------------------------------------
\subsection{Elliptic weights}
%------------------------------------------------------------------------------------------------

For nome $p\in\C$ with $|p|<1$, \emph{base} $q\in\C$, two independent variables
$a$ and $b$, and $(s,t)\in\Z^2$, we define the
\emph{small elliptic weights} to be
\begin{subequations}\label{eqn:ellwt}
\begin{equation}
w_{a,b;q,p}(s,t)=\frac{\theta(aq^{s+2t},bq^{2s+t-2},aq^{t-s-1}/b;p)}
{\theta(aq^{s+2t-2},bq^{2s+t},aq^{t-s+1}/b;p)}q,\label{eqn:ellipticwt}
\end{equation}
and the \emph{big elliptic weights} to be
\begin{equation}
W_{a,b;q,p}(s,t)=\frac{\theta(aq^{s+2t},bq^{2s},bq^{2s-1},
aq^{1-s}/b,aq^{-s}/b;p)}
{\theta(aq^{s},bq^{2s+t},bq^{2s+t-1},aq^{t-s+1}/b,aq^{t-s}/b;p)}q^t.
\label{eqn:ellipticbigwt}
\end{equation}
\end{subequations}
Notice that for $t\ge 0$ we have
\begin{equation*}
W_{a,b;q,p}(s,t)=\prod_{k=1}^t w_{a,b;q,p}(s,k).
\end{equation*}
Observe that
\begin{subequations}
\begin{equation}\label{wshift}
w_{a,b;q,p}(s+i,t+j)=w_{aq^{i+2j},bq^{2i+j};q,p}(s,t),
\end{equation}
and 
\begin{equation}\label{Wshift}
W_{a,b;q,p}(s,t+j)=W_{a,b;q,p}(s,j)\,
W_{aq^{2j},bq^j;q,p}(s,t),
\end{equation}
\end{subequations}
for all $s$, $t$, $i$ and $j$,
which are elementary identities we will make use of.

Further, using \eqref{eq:inv}, we see directly from \eqref{eqn:ellipticwt} that
\begin{equation}\label{winv}
w_{a,b;q,p}(1-s-t,t)^{-1}=w_{a/b,1/b;q,p}(s,t),
\end{equation}
which can be conveniently applied when using Lemma~\ref{lem:inv}. 

The terminology ``elliptic'' for the above small and big weights
is indeed justified, as the small weight $w_{a,b;q,p}(s,k)$ (and also the
big weight) is elliptic in each of its parameters (i.e., these weights are
even ``totally elliptic''). Writing $q=e^{2\pi i\sigma}$,
$p=e^{2\pi i\tau}$, $a=q^\alpha$ and $b=q^\beta$ with complex $\sigma$,
$\tau$, $\alpha$, $\beta$, $s$ and $k$, then the small weight
$w_{a,b;q,p}(s,k)$ is clearly periodic in $\alpha$ with period $\sigma^{-1}$.
Also, using \eqref{eq:qpp}, we can see that
$w_{a,b;q,p}(s,k)$ is also periodic in $\alpha$ with period $\tau\sigma^{-1}$.
The same applies to $w_{a,b;q,p}(s,k)$ as a function in $\beta$ (or $s$ or $k$)
with the same two periods $\sigma^{-1}$ and $\tau\sigma^{-1}$.

Next, we define (cf.\  \cite[Ch.\ 11]{GRhyp})
the {\em theta shifted factorial}
(or {\em $q,p$-shifted factorial}), by
\begin{equation*}
(a;q,p)_n = \begin{cases}
\displaystyle\prod^{n-1}_{j=0} \theta (aq^j;p),& n = 1, 2, \ldots\,,\cr
1,& n = 0,\cr
\displaystyle1/\prod^{-n-1}_{j=0} \theta (aq^{n+j};p),& n = -1, -2, \ldots,
\end{cases}
\end{equation*}
and write
\begin{equation*}
(a_1, \ldots, a_m;q, p)_n = (a_1;q,p)_n\ldots(a_m;q,p)_n,
\end{equation*}
for their products.
For $p=0$ we have  $\theta (x;0) = 1-x$ and, hence, $(a;q, 0)_n = (a;q)_n
=(1-a)(1-aq)\dots(1-aq^{n-1})$
is a {\em $q$-shifted factorial} in base $q$.

Now, the \emph{elliptic binomial coefficients} \cite{Schl1}
\begin{equation}\label{ellbin}
\begin{bmatrix}n\\k\end{bmatrix}_{a,b;q,p}:=
\frac{(q^{1+k},aq^{1+k},bq^{1+k},aq^{1-k}/b;q,p)_{n-k}}
{(q,aq,bq^{1+2k},aq/b;q,p)_{n-k}},
\end{equation}
together with the big elliptic weights defined in \eqref{eqn:ellipticbigwt},
can be seen to satisfy the recursion \eqref{wbineq}, as a consequence
of the addition formula \eqref{eq:addf}.

Note that the elliptic binomial coefficients in \eqref{ellbin}
generalize the familiar $q$-binomial coefficients, which can be obtained
by letting $p\to 0$, $a\to 0$, then $b\to 0$, in this order.
These are defined by 
\begin{equation*}
\begin{bmatrix}n\\k\end{bmatrix}_q:=
\frac{(q^{1+k};q)_{n-k}}
{(q;q)_{n-k}},
\end{equation*}
where
\begin{equation*}
(a;q)_n= 
\begin{cases}
\displaystyle \prod^{n-1}_{j=0} (1-aq^j),& n = 1, 2, \ldots\,,\cr
1,& n = 0,\cr
\displaystyle 1/\prod^{-n-1}_{j=0} (1-aq^{n+j}),& n = -1, -2, \ldots
\end{cases}
\end{equation*}
are the $q$-shifted factorials.
%(Similarly to above, one writes
%$(a_1, \ldots, a_m;q)_n = (a_1;q)_n\ldots(a_m;q)_n,$
%for products for the $q$-shifted factorials.)

As the $q$-binomial coefficients satisfy two recurrence relations
$$
\begin{bmatrix}n+1\\k\end{bmatrix}_q=
\begin{bmatrix}n\\k\end{bmatrix}_q+
\begin{bmatrix}n\\k-1\end{bmatrix}_q q^{n+1-k},\qquad
\begin{bmatrix}n+1\\k\end{bmatrix}_q=
\begin{bmatrix}n\\k\end{bmatrix}_q q^k+
\begin{bmatrix}n\\k-1\end{bmatrix}_q,
$$
and the recurrence relation \eqref{wbineq} corresponds to the first identity,
the elliptic binomial coefficients satisfy a second recurrence relation as well.
While the relation  \eqref{wbineq} is established by considering the
generating function of all weighted paths from the origin to the point $(k,n+1-k)$
and separating them into two subsets depending on whether the last step
is vertical or horizontal, the following result can be similarly verified by
separating the same set of paths into two subsets depending on whether
the first step is vertical or horizontal.
\begin{proposition}\label{prop:2ndrec}
We have
$$
\begin{bmatrix}n+1\\k\end{bmatrix}_{a,b;q,p}=
\begin{bmatrix}n\\k\end{bmatrix}_{aq^2, bq;q,p}\,
\prod_{j=1}^k W_{a,b;q,p}(j,1) +
\begin{bmatrix}n\\k-1\end{bmatrix}_{aq, bq^2;q,p}.$$
\end{proposition}

\begin{definition}\label{def:Cellxy}
Let $x,y,a,b$ be four variables with $ab=ba$ and $q,p$ be two complex numbers
with $|p|<1$. We define $\mathbb{C}_{a,b;q,p}[x,y]$ to be the
unital associative algebra over $\mathbb{C}$, generated by $x$ and $y$,
satisfying the following commutation relations 
\begin{subequations}\label{abqxy}
\begin{align}
yx &= \frac{\theta(aq^3,bq,a/bq;p)}{\theta(aq,bq^3,aq/b;p)}qxy,\label{eqn:xy}\\
x\,f(a,b)&= f(aq,bq^2)x,\label{eqn:x}\\
y\,f(a,b)&=  f(aq^2,bq)y,\label{eqn:y}
\end{align}
\end{subequations}
where $f(a,b)$ is any function that is
multiplicatively $p$-periodic in $a$ and $b$,
i.e., which satisfies $f(pa,b)=f(a,pb)=f(a,b)$.
\end{definition}

The relations in \eqref{abqxy} are essentially an elliptic realization of
the relations in \eqref{eqn:noncommrel}.
In particular, \eqref{eqn:xy} can be written as $yx=w(1,1)xy$ with
$w(s,t)=w_{a,b;q,p}(s,t)$ being the small elliptic weight in
\eqref{eqn:ellipticwt}.

We refer to the variables $x,y,a,b$ forming $\mathbb{C}_{a,b;q,p}[x,y]$
as \emph{elliptic-commuting} variables. The algebra 
$\mathbb{C}_{a,b;q,p}[x,y]$ formally reduces to $\mathbb{C}_q[x,y]$
if one lets $p\to 0$, $a\to 0$, then $b\to 0$ (in this order),
while, having eliminated the nome $p$, relaxing the two conditions
of multiplicative $p$-periodicity.

In $\mathbb{C}_{a,b;q,p}[x,y]$ the following binomial theorem holds
as a consequence of Theorem~\ref{thm:binomial}
(cf.\ \cite{Schl1}):
\begin{equation}\label{eqn:ab_binom}
(x+y)^n =\sum_{k=0}^n
\begin{bmatrix}n\\k\end{bmatrix}_{a,b;q,p}x^k y^{n-k}.
\end{equation}

It is now straightforward to extend  $\mathbb{C}_{a,b;q,p}[x,y]$
in the spirit of Definition~\ref{def:Cwxy2} to an algebra we name
$\mathbb{C}_{a,b;q,p}[x,x^{-1},y]$ (and keep referring to as
algebra of elliptic-commuting variables).
\begin{definition}\label{def:Cellxy2}
Let $x,y,a,b$ be four variables, $x$ invertible, with $ab=ba$ and
$q,p$ be two complex numbers with $|p|<1$.
We define $\mathbb{C}_{a,b;q,p}[x,x^{-1},y]$ to be the
unital associative algebra over $\mathbb{C}$, generated by $x$, $x^{-1}$,
and $y$, satisfying the following commutation relations 
\begin{subequations}\label{abqxy2}
\begin{align}
x^{-1}x&=xx^{-1}=1\\
yx &= \frac{\theta(aq^3,bq,a/bq;p)}{\theta(aq,bq^3,aq/b;p)}qxy,\label{eqn:xy2}\\
x^{-1}y&=\frac{\theta(aq^2,b/q,a/b;p)}{\theta(a,bq,aq^2/b;p)}qyx^{-1},\\
x\,f(a,b)&= f(aq,bq^2)x,\label{eqn:x2}\\
x^{-1}\,f(a,b)&= f(aq^{-1},bq^{-2})x^{-1},\label{eqn:x3}\\
y\,f(a,b)&=  f(aq^2,bq)y,\label{eqn:y2}
\end{align}
\end{subequations}
where $f(a,b)$ is any function that is
multiplicatively $p$-periodic in $a$ and $b$.
\end{definition}
Again, it is not difficult to see that the conditions in \eqref{abqxy2}
are compatible with each other and naturally extend those in \eqref{abqxy}
by adding relations involving $x^{-1}$.

%------------------------------------------------------------------------------------------------
\section{Noncommutative Fibonacci polynomials}\label{sec:nc}
%------------------------------------------------------------------------------------------------

In the following, we shall first assume $x$ and $y$ (which in
Section~\ref{sec:pre} were prescribed to satisfy specific commutation
relations) to be non-commutative variables without any relation
connecting them; we shall only later specialize $x$ and $y$
when explicitly stated.

The material in this section, up to \eqref{eq:nCassini}, is essentially
a review of work done by Johann Cigler and is included here for convenience
and self-containedness.

The \emph{noncommutative} (or \emph{abstract}) Fibonacci polynomials of
Cigler~\cite{C,C2} are defined by $F_0(x,y)=1$, $F_1(x,y)=y$ and
\begin{subequations}\label{eq:Frec}
\begin{align}
F_{n+2}(x,y)&=F_n(x,y) x+ F_{n+1}(x,y) y,\label{eq:Frec-a}\\
\intertext{or equivalently}
F_{n+2}(x,y)&=x\, F_n(x,y)+y\, F_{n+1}(x,y),\label{eq:Frec-b}
\end{align}
\end{subequations}
for all $n\ge 0$.

The equivalence of \eqref{eq:Frec-a} and \eqref{eq:Frec-b}
will be shown later. (See the explanation right after \eqref{eq:sumF}.)
Combinatorially, $F_n (x,y)$ represents the sum of the weights of all possible
ordered tilings of a $1\times n$ board in $1\times 2$ dominoes weighted with $x$ and 
$1\times 1$ squares weighted with $y$. (This also explains why the two
recurrences in \eqref{eq:Frec} are equivalent.)

\begin{example}
\begin{align*}
F_0 (x,y)&= 1;\\
F_1 (x,y)& = y ~; ~ \begin{tikzpicture}[scale=.5]
\draw (0,0)--(1,0)--(1,1)--(0,1)--(0,0)--cycle;
\fill[fill=blue!10!white] (.1, .1) rectangle (.9, .9);
\draw[thick, color=dbluecolor] (.1,.1)--(.1, .9)--(.9, .9)--(.9, .1)--(.1, .1)--cycle;
\node[] at (.5, .5) {\color{blue}$y$};
\end{tikzpicture}\\
F_2(x,y) & = x+ y^2 ~;~
 \begin{tikzpicture}[scale=.5]
\draw (0,0)--(2,0)--(2,1)--(0,1)--(0,0)--cycle;
\fill[fill=red!10!white] (.1, .1) rectangle (1.9, .9);
\draw[thick, color=dredcolor] (.1,.1)--(.1, .9)--(1.9, .9)--(1.9, .1)--(.1, .1)--cycle;
\node[] at (1, .5) {\color{dredcolor}$x$};
\end{tikzpicture}, ~
 \begin{tikzpicture}[scale=.5]
\draw (0,0)--(2,0)--(2,1)--(0,1)--(0,0)--cycle;
\fill[fill=blue!10!white] (.1, .1) rectangle (.9, .9);
\draw[thick, color=dbluecolor] (.1,.1)--(.1, .9)--(.9, .9)--(.9, .1)--(.1, .1)--cycle;
\fill[fill=blue!10!white] (1.1, .1) rectangle (1.9, .9);
\draw[thick, color=dbluecolor] (1.1,.1)--(1.1, .9)--(1.9, .9)--(1.9, .1)--(1.1, .1)--cycle;
\node[] at (.5, .5) {\color{blue}$y$};
\node[] at (1.5, .5) {\color{blue}$y$};
\end{tikzpicture}\\
F_3(x,y) &= xy+yx+y^3 ~;~ \begin{tikzpicture}[scale=.5]
\draw (0,0)--(3,0)--(3,1)--(0,1)--(0,0)--cycle;
\fill[fill=red!10!white] (.1, .1) rectangle (1.9, .9);
\draw[thick, color=dredcolor] (.1,.1)--(.1, .9)--(1.9, .9)--(1.9, .1)--(.1, .1)--cycle;
\node[] at (1, .5) {\color{dredcolor}$x$};
\fill[fill=blue!10!white] (2.1, .1) rectangle (2.9, .9);
\draw[thick, color=dbluecolor] (2.1,.1)--(2.1, .9)--(2.9, .9)--(2.9, .1)--(2.1, .1)--cycle;
\node[] at (2.5, .5) {\color{blue}$y$};
\end{tikzpicture},~
 \begin{tikzpicture}[scale=.5]
\draw (0,0)--(3,0)--(3,1)--(0,1)--(0,0)--cycle;
\fill[fill=red!10!white] (1.1, .1) rectangle (2.9, .9);
\draw[thick, color=dredcolor] (1.1,.1)--(1.1, .9)--(2.9, .9)--(2.9, .1)--(1.1, .1)--cycle;
\node[] at (2, .5) {\color{dredcolor}$x$};
\fill[fill=blue!10!white] (.1, .1) rectangle (.9, .9);
\draw[thick, color=dbluecolor] (.1,.1)--(.1, .9)--(.9, .9)--(.9, .1)--(.1, .1)--cycle;
\node[] at (.5, .5) {\color{blue}$y$};
\end{tikzpicture},~
 \begin{tikzpicture}[scale=.5]
\draw (0,0)--(3,0)--(3,1)--(0,1)--(0,0)--cycle;
\fill[fill=blue!10!white] (.1, .1) rectangle (.9, .9);
\draw[thick, color=dbluecolor] (.1,.1)--(.1, .9)--(.9, .9)--(.9, .1)--(.1, .1)--cycle;
\fill[fill=blue!10!white] (1.1, .1) rectangle (1.9, .9);
\draw[ thick,color=dbluecolor] (1.1,.1)--(1.1, .9)--(1.9, .9)--(1.9, .1)--(1.1, .1)--cycle;
\node[] at (.5, .5) {\color{blue}$y$};
\node[] at (1.5, .5) {\color{blue}$y$};
\fill[fill=blue!10!white] (2.1, .1) rectangle (2.9, .9);
\draw[ thick,color=dbluecolor] (2.1,.1)--(2.1, .9)--(2.9, .9)--(2.9, .1)--(2.1, .1)--cycle;
\node[] at (2.5, .5) {\color{blue}$y$};
\end{tikzpicture}\\
F_4(x,y) &= x^2+xy^2+yxy+y^2x+y^4 ~;~ \begin{tikzpicture}[scale=.5]
\draw (0,0)--(4,0)--(4,1)--(0,1)--(0,0)--cycle;
\fill[fill=red!10!white] (.1, .1) rectangle (1.9, .9);
\draw[thick, color=dredcolor] (.1,.1)--(.1, .9)--(1.9, .9)--(1.9, .1)--(.1, .1)--cycle;
\node[] at (1, .5) {\color{dredcolor}$x$};
\fill[fill=red!10!white] (2.1, .1) rectangle (3.9, .9);
\draw[thick, color=dredcolor] (2.1,.1)--(2.1, .9)--(3.9, .9)--(3.9, .1)--(2.1, .1)--cycle;
\node[] at (3, .5) {\color{dredcolor}$x$};
\end{tikzpicture}, ~
 \begin{tikzpicture}[scale=.5]
\draw (0,0)--(4,0)--(4,1)--(0,1)--(0,0)--cycle;
\fill[fill=red!10!white] (.1, .1) rectangle (1.9, .9);
\draw[thick, color=dredcolor] (.1,.1)--(.1, .9)--(1.9, .9)--(1.9, .1)--(.1, .1)--cycle;
\node[] at (1, .5) {\color{dredcolor}$x$};
\fill[fill=blue!10!white] (2.1, .1) rectangle (2.9, .9);
\draw[thick, color=dbluecolor] (2.1,.1)--(2.1, .9)--(2.9, .9)--(2.9, .1)--(2.1, .1)--cycle;
\fill[fill=blue!10!white] (3.1, .1) rectangle (3.9, .9);
\draw[thick, color=dbluecolor] (3.1,.1)--(3.1, .9)--(3.9, .9)--(3.9, .1)--(3.1, .1)--cycle;
\node[] at (2.5, .5) {\color{blue}$y$};
\node[] at (3.5, .5) {\color{blue}$y$};
\end{tikzpicture}, ~\\
&\qquad\qquad\qquad\qquad
\begin{tikzpicture}[scale=.5]
\draw (0,0)--(4,0)--(4,1)--(0,1)--(0,0)--cycle;
\fill[fill=blue!10!white] (.1, .1) rectangle (.9, .9);
\draw[thick, color=dbluecolor] (.1,.1)--(.1, .9)--(.9, .9)--(.9, .1)--(.1, .1)--cycle;
\fill[fill=red!10!white] (1.1, .1) rectangle (2.9, .9);
\draw[thick, color=dredcolor] (1.1,.1)--(1.1, .9)--(2.9, .9)--(2.9, .1)--(1.1, .1)--cycle;
\node[] at (2, .5) {\color{dredcolor}$x$};
\fill[fill=blue!10!white] (3.1, .1) rectangle (3.9, .9);
\draw[thick, color=dbluecolor] (3.1,.1)--(3.1, .9)--(3.9, .9)--(3.9, .1)--(3.1, .1)--cycle;
\node[] at (3.5, .5) {\color{blue}$y$};
\node[] at (.5, .5) {\color{blue}$y$};
\end{tikzpicture},~
 \begin{tikzpicture}[scale=.5]
\draw (0,0)--(4,0)--(4,1)--(0,1)--(0,0)--cycle;
\fill[fill=blue!10!white] (.1, .1) rectangle (.9, .9);
\draw[thick, color=dbluecolor] (.1,.1)--(.1, .9)--(.9, .9)--(.9, .1)--(.1, .1)--cycle;
\fill[fill=red!10!white] (2.1, .1) rectangle (3.9, .9);
\draw[thick, color=dredcolor] (2.1,.1)--(2.1, .9)--(3.9, .9)--(3.9, .1)--(2.1, .1)--cycle;
\node[] at (3, .5) {\color{dredcolor}$x$};
\fill[fill=blue!10!white] (1.1, .1) rectangle (1.9, .9);
\draw[thick, color=dbluecolor] (1.1,.1)--(1.1, .9)--(1.9, .9)--(1.9, .1)--(1.1, .1)--cycle;
\node[] at (1.5, .5) {\color{blue}$y$};
\node[] at (.5, .5) {\color{blue}$y$};
\end{tikzpicture},~
 \begin{tikzpicture}[scale=.5]
 \draw (0,0)--(4,0)--(4,1)--(0,1)--(0,0)--cycle;
\fill[fill=blue!10!white] (.1, .1) rectangle (.9, .9);
\draw[thick, color=dbluecolor] (.1,.1)--(.1, .9)--(.9, .9)--(.9, .1)--(.1, .1)--cycle;
\fill[fill=blue!10!white] (1.1, .1) rectangle (1.9, .9);
\draw[thick, color=dbluecolor] (1.1,.1)--(1.1, .9)--(1.9, .9)--(1.9, .1)--(1.1, .1)--cycle;
\fill[fill=blue!10!white] (2.1, .1) rectangle (2.9, .9);
\draw[ thick,color=dbluecolor] (2.1,.1)--(2.1, .9)--(2.9, .9)--(2.9, .1)--(2.1, .1)--cycle;
\node[] at (1.5, .5) {\color{blue}$y$};
\node[] at (2.5, .5) {\color{blue}$y$};
\fill[fill=blue!10!white] (3.1, .1) rectangle (3.9, .9);
\draw[ thick,color=dbluecolor] (3.1,.1)--(3.1, .9)--(3.9, .9)--(3.9, .1)--(3.1, .1)--cycle;
\node[] at (3.5, .5) {\color{blue}$y$};
\node[] at (.5, .5) {\color{blue}$y$};
\end{tikzpicture}
\end{align*}
\end{example}

% matrix power 

Let
$$C(x,y):= \left( \begin{matrix}0 & 1\\ x & y\end{matrix}\right).$$
Then the $n$-th power of the matrix $C(x,y)$ can be expressed nicely in terms of
the non-commutative Fibonacci polynomials (as was already shown by
Cigler~\cite[Equation~(3.2)]{C2}). 

\begin{proposition}\label{propCn}
$$C^n (x,y)= \left( \begin{matrix}F_{n-2}(x,y) x & F_{n-1}(x,y)\\
F_{n-1}(x,y)x & F_{n}(x,y)\end{matrix}\right),$$
for $n\ge 2$.
\end{proposition}

\begin{proof}
We proceed by induction.\\
For $n=2$, 
\begin{equation*}
C^2 (x,y)= \left( \begin{matrix}0 & 1\\ x & y\end{matrix}\right)\left( \begin{matrix}0 & 1\\ x & y\end{matrix}\right)=\left( \begin{matrix}x & y\\ yx &x+ y^2\end{matrix}\right)
= \left( \begin{matrix}F_{0}(x,y) x & F_{1}(x,y)\\ F_{1}(x,y)x & F_{2}(x,y)\end{matrix}\right).
\end{equation*}
Suppose that 
$$C^{n-1} (x,y)= \left( \begin{matrix}F_{n-3}(x,y) x & F_{n-2}(x,y)\\
F_{n-2}(x,y)x & F_{n-1}(x,y)\end{matrix}\right)$$
holds for some $n-1\ge 2$. Then
\begin{align*}
C^n (x, y) &= C^{n-1} (x,y) C(x,y)\\
&=  \left( \begin{matrix}F_{n-3}(x,y) x & F_{n-2}(x,y)\\
F_{n-2}(x,y)x & F_{n-1}(x,y)\end{matrix}\right)
\left( \begin{matrix}0 & 1\\ x & y\end{matrix}\right)\\
&=  \left( \begin{matrix}F_{n-2}(x,y) x & F_{n-3}(x,y)x + F_{n-2}(x,y)y \\
F_{n-1}(x,y)x & F_{n-2}(x,y)x+F_{n-1}(x,y)y \end{matrix}\right) \\
&=  \left( \begin{matrix}F_{n-2}(x,y) x & F_{n-1}(x,y)\\
F_{n-1}(x,y)x & F_{n}(x,y)\end{matrix}\right),
\end{align*}
by \eqref{eq:Frec-a}.
(Similary, we could have used $C^n(x,y)=C(x,y)C^{n-1} (x,y)$ instead,
in combination with \eqref{eq:Frec-b}.)
\end{proof}

Since $C^{m+n}(x,y)= C^m (x,y) C^n (x,y)$, we have 
\begin{align*}
C^{m+n}(x,y)
 &= \left( \begin{matrix}F_{m+n-2}(x,y) x & F_{m+n-1}(x,y)\\
 F_{m+n-1}(x,y)x & F_{m+n}(x,y)\end{matrix}\right),\\
\intertext{and}
C^m (x,y) C^n (x,y)
&= \left( \begin{matrix}F_{m-2}(x,y) x & F_{m-1}(x,y)\\
F_{m-1}(x,y)x & F_{m}(x,y)\end{matrix}\right) 
\left( \begin{matrix}F_{n-2}(x,y) x & F_{n-1}(x,y)\\
F_{n-1}(x,y)x & F_{n}(x,y)\end{matrix}\right).
\end{align*}
By comparing the components, we obtain the formula
\begin{equation}\label{eq:sumF}
F_{m+n}(x,y)= F_{m-1}(x,y) \, x \, F_{n-1}(x,y) 
+ F_m (x,y) \, F_n (x,y) .
\end{equation}
We proved \eqref{eq:sumF} using the recurrence \eqref{eq:Frec-a},
which can be recovered from the former by letting $(m,n)\mapsto(n+1,1)$.
However, \eqref{eq:sumF} also includes the recurrence \eqref{eq:Frec-b},
obtained by letting $(m,n)\mapsto(1,n+1)$. This shows that the
two recurrences in \eqref{eq:Frec} are indeed equivalent.

From a combinatorial view, the identity \eqref{eq:sumF} is clear.
A tiling of the $1\times (m+n)$ board can be split into two independent
tilings of lengths $m$ and $n$, or there is a domino in the middle,
right between two independent tilings of respective lengths $m-1$ and $n-1$.

We now use \eqref{eq:sumF} in conjunction with negatively indexed
non-commutative Fibonacci polynomials to obtain a non-commutative
(Euler--)Cassini identity.
In fact, one can simply use the recurrences in \eqref{eq:Frec}  to define 
non-commutative Fibonacci polynomials of negative index.
It turns out that these happen to be polynomials in $x^{-1}$
(so we must assume $x$ to be invertible).
In particular, application of \eqref{eq:Frec} in the negative direction gives
\begin{align*}
F_{-1}(x,y)&=0,\\
F_{-2}(x,y)&=x^{-1},\\
F_{-3}(x,y)&=-x^{-1}yx^{-1},\\
F_{-4}(x,y)&=x^{-2}+x^{-1}yx^{-1}yx^{-1},\\
F_{-5}(x,y)&=-x^{-2}yx^{-1}-x^{-1}yx^{-2}-x^{-1}yx^{-1}yx^{-1}yx^{-1},\\
F_{-6}(x,y)&=x^{-3}+x^{-2}yx^{-1}yx^{-1}+x^{-1}yx^{-2}yx^{-1}\\
&\quad\;+x^{-1}yx^{-1}yx^{-2}+x^{-1}yx^{-1}yx^{-1}yx^{-1}yx^{-1}.
\end{align*}
It is easy to use \eqref{eq:Frec} and induction to prove that
\begin{equation}\label{eq:Fneg}
F_{-n}(x,y)=(-1)^{n}F_{n-2}\big(x^{-1},x^{-1}y\big)\,x^{-1},
\end{equation}
for all integers $n$.
The formula \eqref{eq:Fneg} was also obtained by Cigler
in \cite[Equation after (3.4)]{C2}.

We can also use matrices to arrive at negatively indexed
noncommutative Fibonacci polynomials, namely
$$
C^{-1}(x,y)=\left( \begin{matrix}-x^{-1}y & x^{-1}\\ 1 & 0\end{matrix}\right),
$$
which satisfies $C^{-1}(x,y)C(x,y)=C(x,y)C^{-1}(x,y)=I_2$;
Proposition~\ref{propCn} is easily seen to extend to all integers $n$
with the negatively indexed noncommutative Fibonacci polynomials
(defined recursively by \eqref{eq:Frec}, and which can be expressed by
the non-negatively indexed noncommutative Fibonacci polynomials
by \eqref{eq:Fneg}). This was actually how Cigler arrived at
 \eqref{eq:Fneg} in \cite{C2}.
This shows that \eqref{eq:sumF}, which we originally
proved for positive integers $m$ and $n$, actually holds for all integers
$m$ and $n$.

We now let $m=-n$ in \eqref{eq:sumF} and multiply both sides of the identity
by $(-1)^{n}$ and arrive, after two applications of \eqref{eq:Fneg}
at the \emph{non-commutative Cassini identity} (cf.\ \cite[Section~3]{C2})
\begin{equation}\label{eq:nCassini}
(-1)^n= F_{n-2}\big(x^{-1},x^{-1}y\big)\, x^{-1}\, F_n (x,y)
- F_{n-1}\big(x^{-1},x^{-1}y\big) \, F_{n-1}(x,y) ,
\end{equation}
which is valid for all integers $n$.

More generally, we may take $(m,n)\mapsto(-n,n+k)$
in \eqref{eq:sumF} and multiply both sides of the identity
by $(-1)^{n}$ and arrive, after two applications of \eqref{eq:Fneg},
at the \emph{non-commutative Euler--Cassini identity} 
\begin{align}\label{eq:nEulerCassini}
(-1)^nF_k(x,y)
&= F_{n-2}\big(x^{-1},x^{-1}y\big)\, x^{-1}\, F_{n+k} (x,y)\notag\\
&\quad\,- F_{n-1}\big(x^{-1},x^{-1}y\big) \, F_{n+k-1}(x,y) ,
\end{align}
which is valid for all integers $n$ and $k$.

\begin{remark}
We would like to mention that in the classical case the Cassini identity
is usually obtained by taking the determinants of the $n$-th power of the
Fibonacci matrix
$\begin{smallmatrix}\begin{pmatrix}0&1\\1&1\end{pmatrix}\end{smallmatrix}$
and using the property that the determinant of matrices with commuting entries
is multiplicative. This method to obtain a Cassini identity requires adaptation in the
non-commutative setting since the determinant is in general not multiplicative
if the matrices contain entries that do not commute.
In some special cases (in particular when considering quantum matrix
representations of quantum groups) this can be remedied by suitably
modifying the definition of determinant and by considering
\emph{quantum determinants} instead where the entries of the matrices
obey certain commutation relations.

We can apply this construction with the necessary adaptions here as well.
For four polynomials $a\big(x,x^{-1},y\big)$, $b\big(x,x^{-1},y\big)$
$c\big(x,x^{-1},y\big)$ $d\big(x,x^{-1},y\big)$ in the non-commuting variables
$x$, $x^{-1}$ and $y$
(with $x^{-1}x=xx^{-1}=1$), over some ground field $K$
(say $K=\mathbb C$), define the noncommutative determinant
$\overrightarrow{\det}$ by 
\begin{align}
&\overrightarrow{\det}\begin{pmatrix}a\big(x,x^{-1},y\big)&b\big(x,x^{-1},y\big)\\
c\big(x,x^{-1},y\big)&d\big(x,x^{-1},y\big)\end{pmatrix}\notag\\&:=
a\big(x^{-1},x,x^{-1}y\big)\,d\big(x,x^{-1},y\big)-
c\big(x^{-1},x,x^{-1}y\big)\,x\,b\big(x,x^{-1},y\big).
\end{align}
Now $\overrightarrow{\det}$ is in general not multiplicative
but for suitable choices of the matrices it is. This in particular applies to
matrices given by any integer power of $C(x,y)$, as one can easily verify.
Now taking the noncommutative determinant
$\overrightarrow{\det}$ of $C^n(x,y)$ and comparing it with the
$n$-th power of $\overrightarrow{\det}\, C(x,y)=-1$, we readily
obtain the non-commutative Cassini identity \eqref {eq:nCassini}.
\end{remark}

%------------------------------------------------------------------------------------------------
\subsection{Noncommutative weight-dependent Fibonacci polynomials}
%------------------------------------------------------------------------------------------------
We consider the noncommutative Fibonacci polynomials $F_n(x,y)$
with additional weight-dependent commutation relations imposed
to involve a doubly indexed sequence of invertible weights
$(w(s,t))_{s\in\mathbb Z, t\in\mathbb N}$.
More precisely we shall work in the algebra of weight-dependent variables 
$\mathbb C_w[x,x^{-1},y]$ defined in Definition~\ref{def:Cwxy2}.
We write $F_n(x,y\,|\,w)$ for the respective noncommutative
weight-dependent Fibonacci polynomials in this case.

Any expression $X$ in $\mathbb C_w[x,x^{-1},y]$ can be \emph{normalized}
and written uniquely as a formal sum (with finitely many non-vanishing terms)
$$
X=\sum_{k=-\infty}^ \infty\sum_{\ell=0}^\infty c_X(k,\ell)x^ky^{\ell},
$$
with $c_X(k,\ell)$  a polynomial expression over $\mathbb C$
in the $w(s,t)^{\pm1}$, $s\in\mathbb Z, t\in\mathbb N$.
We say that the variables occurring in an expression in
$\mathbb C_w[x,x^{-1},y]$ have been \emph{normally ordered}
if, as above, in each of the monomials all the occurrences of
$y$ have been moved (with the help of commutation relations,
if necessary) to the most right, followed by the occurrences of
$x$ or $x^{-1}$ to the left (again with the help of commutation relations,
if necessary) and if only to the most left polynomials
of the respective monomials the various weights
$w(s,t)^{\pm1}$ appear.

From the first recurrence relation in \eqref{eq:Frec} and the
recurrence for the weighted binomial coefficients in \eqref{wbineq}
(where the ``big weight" $W(s,k)$ that appears there is a product of
small weights $w(s,t)$, according to \eqref{eq:W}) we can easily prove
the following result for the normally ordered noncommutative
weight-dependent Fibonacci polynomials.

\begin{proposition}\label{prop:binF}
As elements of $\mathbb C_w[x,x^{-1},y]$,
the \emph{noncommutative weight-dependent Fibonacci polynomials}
$F_n(x,y\,|\,w)$ (which are recursively defined by the two intial
values $F_0(x,y\,|\,w)=1$,
$F_1(x,y\,|\,w)=y$, and either one of the two recurrence relations in
\eqref{eq:Frec}) take the following normalized form:
\begin{equation}\label{def:wfib}
F_n (x,y\,|\,w) = \sum_{k=0}^n
{}_{\stackrel{\phantom w}{\stackrel{\phantom w}w}}\!\!
\begin{bmatrix}n-k\\k\end{bmatrix}
 x^k y^{n-2k},
\end{equation}
where ${}_{\stackrel{\phantom w}{\stackrel{\phantom w}w}}\!\!
\begin{smallmatrix}\begin{bmatrix}n\\k\end{bmatrix}\end{smallmatrix}$
is the weight-dependent binomial coefficient recursively
defined in \eqref{wbineq}, for any non-negative integer $n$.
\end{proposition}
\begin{proof}
We proceed by induction. For $n=0$ and $n=1$ \eqref{def:ellfib} is clear.
Now assume that the formula is true for all non-negative integers up to $n+1$.
To show it for the next value, $n+2$, apply the first identity in \eqref{eq:Frec}
to split $F_{n+2}(x,y\,|\,w)$ in two lower-indexed 
noncommutative weight-dependent Fibonacci polynomials and
apply the induction hypothesis. Concretely, we have
\begin{align*}
&F_{n+2}(x,y\,|\,w)\\
&=  F_{n+1}(x,y\,|\,w)\, y + F_n(x,y\,|\,w)\, x\\
&=  \sum_{k=0}^{n+1}
{}_{\stackrel{\phantom w}{\stackrel{\phantom w}w}}\!\!
\begin{bmatrix}n+1-k\\k\end{bmatrix} x^k y^{n+1-2k}\cdot y
+  \sum_{k=1}^{n+1}
{}_{\stackrel{\phantom w}{\stackrel{\phantom w}w}}\!\!
\begin{bmatrix}n-(k-1)\\k-1\end{bmatrix}x^{k-1}y^{n+2-2k}\cdot x\\
&=  \sum_{k=0}^{n+1}
{}_{\stackrel{\phantom w}{\stackrel{\phantom w}w}}\!\!
\begin{bmatrix}n+1-k\\k\end{bmatrix} x^k y^{n+1-2k}\cdot y
\\&\qquad
+  \sum_{k=1}^{n+1}
{}_{\stackrel{\phantom w}{\stackrel{\phantom w}w}}\!\!
\begin{bmatrix}n-(k-1)\\k-1\end{bmatrix}x^{k-1}\,
W(1, n+2-2k) \, x\, y^{n+2-2k}\\
&= \sum_{k=0}^{n+2}\left(
{}_{\stackrel{\phantom w}{\stackrel{\phantom w}w}}\!\!
\begin{bmatrix}n+1-k\\k\end{bmatrix}+
{}_{\stackrel{\phantom w}{\stackrel{\phantom w}w}}\!\!
\begin{bmatrix}n+1-k\\k-1\end{bmatrix}\,  W(k, n+2-2k)  \right)x^k y^{n+2-2k}\\
&= \sum_{k=0}^{n+2}
{}_{\stackrel{\phantom w}{\stackrel{\phantom w}w}}\!\!
\begin{bmatrix}n+2-k\\k\end{bmatrix}x^k y^{n+2-2k},
\end{align*}
where we have applied an instance of Lemma~\ref{lem:com} in the third equality
and the recursion for the weight-dependent binomial coefficients
\eqref{wbineq} in the last equality.
%For the second one, we use the recurrence relation of Proposition \ref{prop:2ndrec};
%\begin{align*}
%&F_{n+2}(x,y,a,b;q,p)\\
%&= \sum_{k=0}^{n+2}\begin{bmatrix}n+2-k\\k\end{bmatrix}_{a,b;q,p} x^k y^{n+2-2k}\\
%&= \sum_{k=0}^{n+2}\left( \begin{bmatrix}n+1-k\\k-1\end{bmatrix}_{aq, bq^2;q,p}+ \begin{bmatrix}n+1-k\\k\end{bmatrix}_{aq^2, bq;q,p}\,  \prod_{j=1}^k W_{a,b;q,p}(j,1)  \right)x^k y^{n+2-2k}\\
%&=  \sum_{k=0}^{n+1}\begin{bmatrix}n+1-k\\k-1\end{bmatrix}_{aq,bq^2;q,p} x^k y^{n+2-2k} 
% +  \sum_{k=0}^{n+1}\begin{bmatrix}n+1-k\\k\end{bmatrix}_{aq^2,bq;q,p}  \prod_{j=1}^k w_{a,b;q,p}(j,1)  x^k y^{n+2-2k}\\
%&= x  \sum_{k=0}^{n+1}\begin{bmatrix}n+1-k\\k-1\end{bmatrix}_{a,b;q,p} x^{k-1} y^{n+2-2k}
%+  \sum_{k=0}^{n+1}\begin{bmatrix}n+1-k\\k\end{bmatrix}_{aq^2,bq;q,p}   y   x^k y^{n+1-2k}\\
%&= x  \sum_{k=1}^{n+1}\begin{bmatrix}n-(k-1)\\k-1\end{bmatrix}_{a,b;q,p} x^{(k-1)} y^{n-2(k-1)}
%+ y  \sum_{k=0}^{n+1}\begin{bmatrix}n+1-k\\k\end{bmatrix}_{a,b;q,p}     x^k y^{n+1-2k}\\
%&= x\, F_n(x,y,a,b;q,p)+y\, F_{n+1}(x,y,a,b;q,p).
%\end{align*}
%Note that from the third to the fourth line, we have used
%$W_{a,b;q,p}(j,1)=w_{a,b;q,p}(j,1)$ for $j\ge 1$,
%while in the fourth to fifth line we have used \chk.
\end{proof}

A great deal of the analysis from the beginning of this section
which concerned the noncommutative Fibonacci polynomials
extends to the noncommutative weight-dependent case without much changes. 
This in particular concerns the formula \eqref{eq:sumF}
which now obviously takes the form
\begin{equation}\label{eq:sumwF}
F_{m+n}(x,y\,|\,w)= F_{m-1}(x,y\,|\,w) \, x \, F_{n-1}(x,y\,|\,w) 
+ F_m (x,y\,|\,w)\, F_n (x,y\,|\,w) .
\end{equation}
Again, this formula holds for all integers $m$ and $n$ but
we still have to specify the exact form of the negatively indexed
noncommutative weight-dependent Fibonacci polynomials.
By carrying out the same analysis that led to \eqref{eq:Fneg}
(i.e., application of the recurrence \eqref{eq:Frec} in the negative
direction, and induction) adapted to the weight-dependent setting, we obtain
\begin{equation}\label{eq:wFneg}
F_{-n}(x,y\,|\,w)=(-1)^{n}F_{n-2}\big(x^{-1},x^{-1}y\,|\,\widetilde{w}\big)\,x^{-1},
\end{equation}
where the dual weight function $\widetilde{w}$ is defined in \eqref{eq:tw}.

The noncommutative weight-dependent Euler--Cassini identity
thus takes the following form:
\begin{align}\label{eq:nwEulerCassini}
(-1)^n F_k (x,y\,|\,w)&= F_{n-2}\big(x^{-1},x^{-1}y\,|\,\widetilde{w}\big)\,
x^{-1}\, F_{n+k} (x,y\,|\,w)\notag\\
&\quad\,- F_{n-1}\big(x^{-1},x^{-1}y\,|\,\widetilde{w}\big)
\, F_{n+k-1}(x,y\,|\,w) ,
\end{align}
which is valid for all integers $n$ and $k$.

%------------------------------------------------------------------------------------------------
\subsection{Noncommutative elliptic Fibonacci polynomials}
%------------------------------------------------------------------------------------------------
We now specialize the weights $w(s,t):=w_{a,b;q,p}(s,t)$
(for $s\in\mathbb Z$ and $t\in\mathbb N$) to be the elliptic weights
defined in \eqref{eqn:ellipticwt},
where $a,b$ are two independent parameters and
$p,q$ are  complex numbers with $|p|<1$. We are thus working
in the algebra of elliptic-commuting variables
$\mathbb C_{a,b;q,p}[x,x^{-1},y]$ defined in Definition~\ref{def:Cellxy2}.
We write $F_n(x,y\,|\,a,b;q,p)$ for the respective noncommutative
elliptic Fibonacci polynomials in this case.

Specialization of Proposition~\ref{prop:binF} readily gives the following result.
\begin{corollary}\label{cor:binellF}
As elements of $\mathbb C_{a,b;q,p}[x,x^{-1},y]$,
the \emph{noncommutative elliptic Fibonacci polynomials}
$F_n(x,y\,|\,a,b;q,p)$ take the following normalized form:
\begin{equation}\label{def:ellfib}
F_n (x,y\,|\,a,b;q,p) = \sum_{k=0}^n \begin{bmatrix}n-k\\k\end{bmatrix}_{a,b;q,p} x^k y^{n-2k},
\end{equation}
where $\begin{smallmatrix}\begin{bmatrix}n\\
k\end{bmatrix}\end{smallmatrix}_{a,b;q,p}$
is the elliptic binomial coefficient given in \eqref{ellbin},
for any non-negative integer $n$.
\end{corollary}

Now, the specialization of \eqref{eq:sumwF} is straightforward and gives
\begin{align}\label{eq:sumellF}
F_{m+n}(x,y\,|\,a,b;q,p)&= F_{m-1}(x,y\,|\,a,b;q,p) \, x
\, F_{n-1}(x,y\,|\,a,b;q,p) \notag\\
&\quad\;+ F_m (x,y\,|\,a,b;q,p)\, F_n (x,y\,|\,a,b;q,p),
\end{align}
which again holds for all integers $m$ and $n$.

Finally, we determine the exact form of the negatively indexed 
noncommutative elliptic Fibonacci polynomials.
Combination of \eqref{eq:tw} and \eqref{winv} gives
the following formula for the \emph{dual weights}
\begin{equation}
\widetilde{w}_{a,b;q,p}(s,t)=w_{a/b,1/b;q,p}(s,t),
\end{equation}
which shows that the negatively indexed 
noncommutative elliptic Fibonacci polynomials can be
conveniently written in terms of the non-negatively indexed ones.
We thus have 
\begin{equation}\label{eq:ellFneg}
F_{-n}(x,y\,|\,a,b;q,p)=(-1)^{n}
F_{n-2}\big(x^{-1},x^{-1}y\,|\,a/b,1/b;q,p\big)\,x^{-1}.
\end{equation}
The noncommutative elliptic Euler--Cassini identity thus takes the following form:
\begin{align}\label{eq:nellCassini}
&(-1)^n F_k (x,y\,|\,a,b;q,p)\notag\\
&=F_{n-2}\big(x^{-1},x^{-1}y\,|\,a/b,1/b;q,p\big)\,
x^{-1}\, F_{n+k} (x,y\,|\,a,b;q,p)\notag\\
&\quad\;- F_{n-1}\big(x^{-1},x^{-1}y\,|\,a/b,1/b;q,p\big) \,
F_{n+k-1}(x,y\,|\,a,b;q,p) ,
\end{align}
which is valid for all integers $n$ and $k$.

\section*{Acknowledgements}
The authors would like to thank Be\'ata B\'enyi for fruitful discussions.

%%%%%%%%%%%%%%%%%%%%%%%%%%%%%


\begin{thebibliography}{99}

%\bibitem{BP} A.T.\ Benjamin and S.S.\ Plott,
%\emph{A combinatorial approach to Fibonomial coefficients},
%Fibonacci Quart.\ \textbf{46/47}(1) (2008/09) 7--9.

%\bibitem{BCMS} C.\ Bennet, J.\ Carillo, J.\ Machacek and B.E.\ Sagan,
%\emph{Combinatorial interpretations of Lucas analogues of
%binomial coefficients and Catalan numbers},
%\texttt{arXiv:1809.09036}.

\bibitem{BCK} N.~Bergeron, C.~Ceballos, and J.~K\"ustner,
``Elliptic and $q$-analogs of the Fibonomial numbers",
{\em SIGMA} \textbf{16} (2020) 076, 16 pp.

\bibitem{C} J.~Cigler,
``$q$-Fibonacci polynomials",
{\em Fibonacci Quart.\ }\textbf{41} (2003), 31--40.

\bibitem{C2} J.~Cigler,
``Some algebraic aspects of Morse code sequences",
{\em Discrete Math.\ Theoret.\ Comp.\ Sci.\ }\textbf{6} (2003), 55--68.

\bibitem{GRhyp} G.~Gasper and M.~Rahman,
{\em Basic hypergeometric series}, second edition,
Encyclopedia of Mathematics and Its Applications~\textbf{96},
Cambridge University Press, Cambridge, 2004.

%\bibitem{GV} I.\ Gessel and X.\ Viennot,
%\emph{Binomial determinants, paths, and hook length formulae},
%Adv.\ Math.\ \textbf{58}(3) (1985), 300--321.

%\bibitem{GR} A.~M.~Garsia and J.~B.~Remmel,
%``$Q$-counting rook configurations and a formula of Frobenius'',
%{\em J.\ Combin.\ Theory Ser.\ A} \textbf{41} (1986), 246--275.

\bibitem{L1} E.\ Lucas,
``Theorie des fonctions numeriques simplement periodique",
{\em Amer.\ J.\ Math.\ }\textbf{1}(2) (1878), 184--196.

\bibitem{L2} E.\ Lucas,
``Theorie des fonctions numeriques simplement periodique",
{\em Amer.\ J.\ Math.\ }\textbf{1}(3) (1878), 197--240.

\bibitem{L3} E.\ Lucas,
``Theorie des fonctions numeriques simplement periodique",
{\em Amer.\ J.\ Math.\ }\textbf{1}(4) (1878), 289--321.

%\bibitem{MS1} T.~Mansour and M.~Schork,
%``Commutation Relations, Normal Ordering, and Stirling Numbers'',
%Chapman and Hall/CRC Press, 2015.

%\bibitem{SS} B.E.\ Sagan and C.D.\ Savage,
%\emph{Combinatorial interpretations of binomial coefficient analogues related to
%Lucas sequences}, Integers \textbf{10:A52} (2010), 697--703.

\bibitem{Schl0}  M.~J.~Schlosser,
``Elliptic enumeration of nonintersecting lattice paths'',
{\em J.\ Combin.\ Theory Ser.\ A} \textbf{114} (2007), 505--521.

\bibitem{Schl1}  M.~J.~Schlosser,
``A noncommutative weight-dependent generalization of the binomial theorem'',
{\em S\'em.\ Lothar.\ Combin.\ }\textbf{B81j} (2020), 24 pp.

\bibitem{SY2} M.~J.~Schlosser and M.~Yoo,
``Some combinatorial identities involving noncommuting variables'',
Proceedings of the Conference
``Formal Power Series and Algebraic Combinatorics'', KAIST, Daejeon,
South Korea, 2015; DMTCS proc.\ FPSAC'15, 2015, 961--972.

\bibitem{SY-rook} M.~J.~Schlosser and M.~Yoo, 
``Elliptic rook and file numbers'',
{\em Electronic J.\ Combin.\ } \textbf{24}(1) (2017), \#P1.31, 47 pp. 

\bibitem{SY-com} M.\ Schlosser and M.\ Yoo,
``Weight-dependent commutation relations and combinatorial identities",
{\em Discrete Math.\ }\textbf{341} (2018), 2308--2325.

\bibitem{W} H.\ Weber,
\textit{Elliptische Functionen und Algebraische Zahlen}, Vieweg-Verlag,
Braunschweig, 1897.

\bibitem{WhW} E.T.~Whittaker and G.N.~Watson, {\em A Course of
Modern Analysis}, 4th ed., Cambridge University Press, Cambridge, 1962.


\end{thebibliography}
\end{document}